\documentclass[12pt]{amsart}
\usepackage{amssymb,amsmath,amsthm,bm}
\usepackage[colorinlistoftodos,prependcaption,textsize=tiny]{todonotes}
 \definecolor{darkblue}{RGB}{0,0,160}
\usepackage{fouriernc} 
\usepackage[colorlinks=true,allcolors=darkblue]{hyperref}
\DeclareSymbolFont{usualmathcal}{OMS}{cmsy}{m}{n}
\DeclareSymbolFontAlphabet{\mathcal}{usualmathcal}

\usepackage[T1]{fontenc}

\usepackage{color}

\usepackage{parskip}
\usepackage{setspace}

\newtheorem{theorem}{Theorem}[section]

\newtheorem{lemma}[theorem]{Lemma}

\newtheorem{conjecture}[theorem]{Conjecture}
\newtheorem*{definition*}{Definition}

\begin{document}

\title{Radial projection theorems in finite spaces}
\author{Ben Lund\and Thang Pham\and Vu Thi Huong Thu}

\address{Discrete Mathematics Group, Institute for Basic Science (IBS)}
\email{benlund@ibs.re.kr}
\address{Hanoi University of Science, Vietnam National University}
\email{thangpham.math@vnu.edu.vn}

\address{Hanoi University of Science, Vietnam National University}
\email{vuthihuongthu\_t64@hus.edu.vn}

\date{}
\maketitle
\begin{abstract}
Motivated by recent results on radial projections and applications to the celebrated Falconer distance problem, we study radial projections in the setting of finite fields. More precisely, we extend results due to Mattila and Orponen (2016), Orponen (2018), and Liu (2020) to finite spaces. In some cases, our results are stronger than the corresponding results in the continuous setting.
In particular, we solve the finite field analog of a conjecture due to Liu and Orponen on the exceptional set of radial projections of a set of dimension between $d-2$ and $d-1$. 
\end{abstract}
\section{Introduction}

For any point $y \in \mathbb{R}^d$, the radial projection $\pi^y:\mathbb{R}^d \setminus \{y\} \rightarrow S^{d-1}$ is
\[\pi^y(x) = \frac{x-y}{|x-y|}. \]
For $E \subset \mathbb{R}^d$, the radial projection of $E$ from $y$ is
\[\pi^y(E) = \{\pi^y(x) : x \in E \setminus \{y\}\}.\]
There are many results and open problems related to the general question: What can be said about the radial projections $\{\pi^y(E): y \in \mathbb{R}^d\}$, given a bound on the Hausdorff dimension $\dim_H(E)$ of $E$?

Here we consider finite analogs to various results on this general question.
Over a general field $\mathbb{F}$, for a given point $y \in \mathbb{F}^d$, the radial projection $\pi^y$ from $y$ maps a point $x \neq y$ to the line that contains $x$ and $y$.
Usually we work over a finite field $\mathbb{F}_q$, and in this setting the standard analog to $\dim_H(E)$ is $\log_q(|E|)$.


One reason to be interested in finite analogs is that working in the finite field model can help to understand the behavior of the corresponding problems in the continuous setting.
A second reason that we are interested in radial projections of point sets is because of a connection to the Falconer distance problem.
Recent work on the Falconer distance problem \cite{Du3, alex-fal} relies on a radial projection theorem due to Orponen \cite{o1}.
We hope that a better understanding of radial projections of finite point sets may also lead to progress on finite versions of the Falconer distance problem - see the last paragraph for a more detailed discussion.

Mattila and Orponen \cite{MO16} proved that for any Borel set $E\subset \mathbb{R}^d$, if $\dim_H(E)>d-1$, then
\begin{equation}\label{conti-1}\dim_H\{y\in \mathbb{R}^d\colon \mathcal{H}^{d-1}(\pi^y(E))=0\}\le 2(d-1)-\dim_H(E),\end{equation}
where $\mathcal{H}^{d}$ denotes the $d$-dimensional Hausdorff measure.
Orponen showed \cite{o1} that this is sharp in the sense that for any $\alpha\in (d-1, d]$, there exists a Borel set $E\subset \mathbb{R}^d$ of Hausdorff dimension $\alpha$ such that $\dim_H\{y: \mathcal{H}^{d-1}(\pi^y(E)) = 0\} = 2(d-1) - \alpha$.

Our first result is an analogous bound for sets in finite space.
\begin{theorem}\label{th:largeESC}
Let $E \subset \mathbb{F}_q^d$ and let $M$ be a positive integer, with $|E| \geq 6q^{d-1}$ and $M \leq 4^{-1}q^{d-1}$.
Then,
\begin{equation} \label{eq:largeESC}
\#\{y\in \mathbb{F}_q^d\colon |\pi^y(E)|\leq M\} <  12q^{d-1}M|E|^{-1}.\end{equation}
\end{theorem}
This is stronger than the most natural finite analog to (\ref{conti-1}) in the case that $M$ is much smaller than $q^{d-1}$, and suggests the following strengthening of (\ref{conti-1}).
\begin{conjecture}\label{conj12}
Let $E \subset \mathbb{R}^d$ be a Borel set with $\dim_H(E) > d-1$, and let $s\leq d-1$ be a real number.
Then,
\[\dim_H\{y \in \mathbb{R}^d: \dim_H\pi^y(E)< s\} \leq d-1+s-\dim_H(E). \]
\end{conjecture}

Liu \cite{liu} proved that for any Borel set $E\subset \mathbb{R}^d$ with $\dim_H(E) \leq d-1$,
\begin{equation}\label{conti-2}
    \dim_H\{y\in \mathbb{R}^d\colon \dim_H \pi^y(E)<\dim_H(E)\}\le \min(2(d-1)-\dim_H(E), \dim_H(E)+1).
\end{equation}
Liu and Orponen \cite{liu} further suggested the following, stronger conjecture.
\begin{conjecture}\label{conjec}
Let $E\subset \mathbb{R}^d$ be a Borel set with $\dim_H(E)\in (k-1, k]$ for $k \in \{0, 1,\ldots, d-1\}$. Then 
\[\dim_H\left\lbrace y\in \mathbb{R}^d\colon \dim_H\pi^y(E)<\dim_H(E) \right\rbrace\le k. \]
\end{conjecture}

We prove the natural finite field analog to Conjecture \ref{conjec} for $|E| \geq q^{d-2}$.
\begin{theorem}\label{th:largeTSC}
Let $E \subset \mathbb{F}_q^d$ with $|E| \leq 100^{-1}q^{d-1}$.
Then,
\begin{equation}\label{eq:largeTSC}\#\{y \in \mathbb{F}_q^d: |\pi^y(E )| \leq 10^{-1}|E|\} < 8q^{d-1}. 
\end{equation}
\end{theorem}

For $|E| < q^{d-2}$, we prove the following bound, which is slightly stronger than the natural analog to (\ref{conti-2}) over the corresponding range.
\begin{theorem}\label{th:justCS-intro}
Let $E \subset \mathbb{F}_q^d$ and let $1<C<|E|$.
Then,
\begin{equation}\label{eq:justCS}
\#\{y \in \mathbb{F}_q^d:|\pi^y(E)| < C^{-1}|E|\} < (C-1)^{-1}q|E|. \end{equation}
\end{theorem}

We do not believe that Theorem \ref{th:justCS-intro} is sharp, and offer the following conjecture, which is a natural finite analog to Conjecture \ref{conjec}.
\begin{conjecture}\label{conj:ff}
Let $E \subset \mathbb{F}_q^d$ with $q^{k-1} < |E| \leq q^k$ for $k \in \{1,\ldots,d-1\}$.
Then,
\[\#\{y \in \mathbb{F}_q^d: |\pi^y(E)| < 10^{-1}|E|\} \leq 10q^k . \]
\end{conjecture}

If true, then Conjectures \ref{conjec} and \ref{conj:ff} would be best possible, up to constant factors.
Indeed, if $E$ is contained in a translate of a $k$-dimensional subspace $F$, then $\dim_H \pi^y(E) \leq k-1$ for every $y \in F$.
However, in some cases it is possible to obtain stronger conclusions by assuming that $E$ is not nearly contained in a translate of a low-dimensional subspace.
For example, Orponen showed that, if $E \subset \mathbb{R}^2$ is a Borel set such that $\dim_H(E \setminus \ell) = \dim_H(E)$ for each line $\ell$, then
\begin{equation}\label{conti-4}\dim_H\left\lbrace y\in \mathbb{R}^2\colon \dim_H \pi^y(E)<\frac{\dim_H(E)}{2} \right\rbrace=0.\end{equation}
As a consequence of this result, when $\dim_H(E)>0$ and $\dim_H(E\setminus \ell) = \dim_H(E)$ for each line $\ell$, there exists $y\in E$ such that $\dim_H \pi^y(E)>\frac{\dim_H(E)}{2}$. 
Several improvements have been obtained in \cite{LIU, Shmerkin, Shmerkin2}. 
The most recent, due to Shmerkin and Wang \cite{Shmerkin2}, states that for a fixed $k\in \{1, \ldots, d-1\}$, $d\ge 2$, and a Borel set $E\subset \mathbb{R}^d$ with $\dim_H(E)=s\in (k-1/k-\eta(d), k]$, where $\eta(d)$ is a small positive constant satisfying $\eta(1)=0$, then 
\begin{equation}\label{best-current}
\sup_{y\in E} \dim_H \pi^y(E) \ge  k-1+\phi(s-k+1),~~\phi(u)=\frac{u}{2}+\frac{u^2}{2(2+\sqrt{u^2+4})},
\end{equation}
under the condition that $E$ is not contained in any $k$-plane. Hence, with $d=2$ and $k=1$, one has there exists $y\in E$ such that 
\begin{equation}\label{best-current-plane}\dim_H\pi^y(E)\ge \phi(s)>s/2.\end{equation}

The following finite analog to (\ref{conti-4}) has a very simple proof, using only the fact that there is a unique line through any pair of points.
\begin{theorem}\label{prop1.6}
Let $\mathbb{F}$ be an arbitrary field, and let $E \subset \mathbb{F}^d$ be a finite set of points such that no line contains more than $(3/4)|E|$ points of $E$.
Then,
\begin{equation}
    \#\{y \in \mathbb{F}^d : |\pi^y(E)|< 2^{-1}|E|^{1/2}\} \leq 1.
\end{equation}
\end{theorem}


In a general setting, we have the next theorem. 
\begin{theorem}\label{th:3.3intro}
Let $\mathbb{F}$ be an arbitrary field and $E$ be a set in $\mathbb{F}^d$ such that no line contains more than $|E|/2$ points of $E$. For $M < |E|/4$, we have
\[\#\{y \in \mathbb{F}^d: |\pi^y(E)| < M\} < 4M^2. \]
\end{theorem}

This theorem is essentially sharp up to constant factors. For example, suppose that $\mathbb{F}_q$ has a subfield of order $p$, and that $E \subset \mathbb{F}_q^2$ is the set of points in a subplane isomorphic to $\mathbb{F}_p^2$.
Then, $|\pi^y(E)| = p+1 > |E|^{1/2}$ for every point $y \in E$. 

We can reasonably hope to improve Theorem \ref{th:3.3intro} under the additional assumption that $\mathbb{F}$ does not have a subfield of suitable size.
We obtain some such improvements using incidence bounds, in particular the Szemer\'edi-Trotter theorem in the real plane \cite{Szemeredi} and an analog proved by Stevens and de Zeeuw for planes over fields of finite characteristic \cite{frank}.

\begin{theorem}\label{radial-real-intro}
There exist constants $0<c_1,c_2<1$ such that the following holds.
For any finite $E \subset \mathbb{R}^2$ such that $|\ell \cap E| < c_1 |E|$ for each line $\ell$ and $M \leq c_2 |E|$, we have
\[\#\{y \in \mathbb{R}^2 : |\pi^y(E)| < M\} < O(M^2 |E|^{-1}). \]
\end{theorem}

One corollary to Theorem \ref{radial-real-intro} (in the spirit of (\ref{best-current-plane})) is that, for any finite set $E \subset \mathbb{R}^2$ with no more than $c_1|E|$ points contained in any line, there is a point $y \in E$ such that $|\pi^y(E)| = O(|E|)$.

\begin{theorem}\label{primefields2D} There exists a constant $0<c<1$ such that the following holds. 
Let $p$ be prime.
For any set $E\subset \mathbb{F}_p^2$ with $|E|\ll p^{8/5}$ and $|\ell\cap E|<c|E|$ for each line $\ell$ and $M\leq c^{11/7}|E|^{4/7}$, we have 
\[\#\{y \in \mathbb{F}_p^2 : |\pi^y(E)| < M\} < O(M^{11/4}|E|^{-1}).\]
\end{theorem}

We conjecture that the conclusion of Theorem \ref{radial-real-intro} holds for sets $E \subset \mathbb{F}_p^2$ with $|E| \ll p$ and no more than $c|E|$ points in any line.

{\bf The Falconer distance problem:} Another motivation of this work that we want to emphasize here is a connection between this topic and the Falconer distance problem, which is one of central problems in Geometric Measure Theory. Given a compact set $E$ in $\mathbb{R}^d$, the Falconer distance conjecture states that if the Hausdorff dimension of $E$ is greater than $d/2$, then its distance set has positive Lebesgue measure. The recent breakthrough of Guth, Iosevich, Ou, and Wang \cite{alex-fal} says that in two dimensions the conclusion holds when $\dim_H(E)>5/4$. In higher dimensions, the best current dimensional thresholds are as follows:
\begin{itemize}
\item $d=3$, Du, Guth, Ou, Wang, Wilson, and Zhang \cite{Du1} (2017): $\frac{3}{2}+\frac{3}{10}$
\item $d\ge 4$ even, Du, Iosevich, Ou, Wang, and Zhang \cite{Du3} (2020): $\frac{d}{2}+\frac{1}{4}$
\item $d\ge 5$ odd, Du and Zhang \cite{Du2} (2018): $\frac{d}{2}+\frac{d}{4d-2}$
\end{itemize}
The finite field variant of this problem is known as the Erd\H{o}s-Falconer distance problem, which asks for the smallest exponent $\alpha$ such that for any $E\subset \mathbb{F}_q^d$, if $|E|\ge q^{\alpha}$, then the distance set covers a positive proportion of all distances. In 2005, Iosevich and Rudnev \cite{IR07} showed that if $|E|\gg q^{\frac{d+1}{2}}$, $E$ determines a positive proportion of all distances. Unlike the continuous version, the exponent $\frac{d+1}{2}$ has been proved to be sharp in \cite{HIKR10} except for even dimensions or dimensions $d\equiv 3\mod 4$ with $q\equiv 3\mod 4$.
For these cases, the conjectured exponent is $d/2$, which is in line with the Falconer distance conjecture. Compared to the continuous version, only a few improvements has been able to make during the last 15 years. In particular, Chapman, Erdogan, Hart, Iosevich, and Koh \cite{CEHIK10} proved the exponent $\frac{4}{3}$ over arbitrary finite fields, which was recently improved to $\frac{5}{4}$ over prime fields by Murphy, Petridis, Pham, Rudnev, and Stevens \cite{murphy}. However, in higher even dimensions, we know nothing except the exponent $(d+1)/2$. Since one of the key steps in the papers \cite{Du3, alex-fal} is a radial projection theorem due to Orponen \cite{o2}, if one wishes to adapt the methods from $\mathbb{R}^d$ to $\mathbb{F}_q^d$, then a full understanding about the radial projection over finite fields is needed. Thus, the results in this paper provide some new progress in this direction. 

{\bf An update on Conjectures \ref{conj12} and \ref{conjec}:} In a very recent paper \cite{OS}, Orponen and Shmerkin proved that Conjecture \ref{conj12} and Conjecture \ref{conjec} hold in two dimensions. We refer the reader to their paper for more details and discussions.

\section{Notation and definitions}
	Here we fix notation and definitions that we use in several proofs.
	
	For $E\subset \mathbb{F}_q^d$ and an integer $M>0$, let
	\[T := \{y \in \mathbb{F}_q^d: |\pi^y(E | \leq M\}. \]
	For any line $\ell$, let $e(\ell) = |\ell \cap E|$, let $t(\ell) = |\ell \cap T|$, and let $\ell(x)$ be the indicator function for $x \in \ell$.
	Let $L$ be the set of lines that each contain at least one point of $T$ and at least one point of $E$, and let $G$ be the set of all lines in $\mathbb{F}_q^d$. We also use the notation $\binom{d}{1}_q$, which is defined by 
	\[\binom{d}{1}_q:=\frac{q^d-1}{q-1}.\]
	Note that $\binom{d}{1}$ is the number of lines in $\mathbb{F}_q^d$ that are incident to a specified point.

\section{Proofs of Theorem \ref{th:justCS-intro}, Theorem \ref{prop1.6}, and Theorem \ref{th:3.3intro}}

The proofs in this section use the Cauchy-Schwarz inequality together with the fact that there is a unique line through each pair of distinct points.

The following lemma immediately implies Theorem \ref{th:justCS-intro}, and plays an important role in the proof of Theorem \ref{th:3.3intro}.

\begin{lemma}\label{lem:TOffALine}
	Let $E \subset \mathbb{F}^d$ be a finite set of points, let $1 < C < |E|$, let $k$ be a positive integer, and let $T$ be a set of points such that
	\begin{itemize}
		\item for each point $y \in T$, we have $|\pi^y(E)| < |E|C^{-1}$, and
		\item for any line $\ell$, we have $|\ell \cap T| < k$.
	\end{itemize}
	Then,
	\[|T| \leq (C-1)^{-1}k|E|. \]
\end{lemma}
\begin{proof}
		The strategy is to find upper and lower bounds on the number of triples in 
	\[ R := \{(y,x_1,x_2) \in T \times E \times E: y,x_1,x_2 \text{ are co-linear and } y \notin \{x_1, x_2\}\}.\]
	Denote by $M = |E|C^{-1}$ the assumed upper bound on $\pi^y(E)$ for $y \in T$.
	To obtain a lower bound on $|R|$, we bound the number of triples in $R$ that contain a fixed $y \in T$:
	\[ \sum_{\substack{\ell: y \in \ell}} (e(\ell) - \mathbf{1}_{y \in E})^2 \geq M^{-1} \left(\sum_{\substack{\ell: y \in \ell}} (e(\ell) - \mathbf{1}_{y \in E}) \right)^2 \geq M^{-1}(|E|-M\cdot \mathbf{1}_{y \in E})^2.\]
	Summing over all $y \in T$:
	\[|R| = \sum_{y \in T}\sum_{\ell: y\in\ell} (e(\ell)-\mathbf{1}_{y \in E})^2 \geq \sum_{y \in T} M^{-1}(|E|- M \cdot \mathbf{1}_{y \in E})^2 = C|T||E| - 2|E||T \cap E| + M|T\cap E|. \] 
	
	The upper bound relies on the observation that each pair of distinct points in $E$ belongs to at most $k$ triples in $R$:
	\begin{align*}
	|R| &= \sum_{x \in E} \sum_{\substack{y \in T \\ x \neq y}} 1 + \sum_{\substack{x_1,x_2 \in E\\ x_1 \neq x_2}} \sum_{\substack{y \in T \\ y \notin \{x_1,x_2\}}} \sum_{\ell \in G}\ell(x_1)\ell(x_2)\ell(y)
	\\ &\leq |T|\,|E| - |E \cap T|  + \sum_{\substack{x_1,x_2 \in E\\ x_1 \neq x_2}} \sum_{\ell \in G} \ell(x_1)\ell(x_2) \sum_{\substack{y \in T \\ y \notin \{x_1,x_2\}}} \ell(y) \\&= |T|\,|E| - |E \cap T|  + \sum_{\substack{x_1,x_2 \in E\\ x_1 \neq x_2}} \sum_{\ell \in G} \ell(x_1)\ell(x_2)(t(\ell)- \mathbf{1}_{x_1 \in T} - \mathbf{1}_{x_2 \in T})
	\\ &\leq |T|\,|E| - |E \cap T| +k|E|^2 - \sum_{\substack{x_1,x_2 \in E\\ x_1 \neq x_2}}  (\mathbf{1}_{x_1 \in T} + \mathbf{1}_{x_2 \in T}) \\
	&= |T|\,|E| - |E \cap T| + k|E|^2 - 2(|E \cap T|(|E \cap T| - 1) + (|E|-|E \cap T|)|E \cap T|)\\
	&= |T|\,|E| + k|E|^2 - 2 |E| \, |E \cap T| + |E \cap T|.\end{align*}
	Combining the lower and upper bounds, we get
	\begin{align*}C|T|\,|E|-2|E|\,|T\cap E| + M|T\cap E| &\leq |T|\,|E| + k|E|^2 - 2|E|\,|E\cap T| + |E \cap T|,\end{align*}
	hence $|T| < (C-1)^{-1}k|E|$.
\end{proof}
As stated above, Lemma \ref{lem:TOffALine} immediately implies theorem \ref{th:justCS-intro}.

\begin{proof}[Proof of Theorem \ref{th:justCS-intro}]
	At most $q$ points of $T$ are contained in any line in $\mathbb{F}_q^d$.
	Hence, Theorem \ref{th:justCS-intro} follows from the case $k=q$ of Lemma \ref{lem:TOffALine}.
\end{proof}

Next we prove Theorem \ref{th:3.3intro}.
Lemma \ref{lem:TOffALine} gives a bound on $|T|$ when not too many points of $T$ are on any single line.
The next lemma bounds the number of elements of $T$ that can be contained in a single line.

\begin{lemma}\label{lem:TOnALine}
	Let $\mathbb{F}$ be a field, and let $E \subset \mathbb{F}^d$ be a finite set of points.
	Let $\ell$ be a line such that $\ell \cap E = \emptyset$.
	Then, for $M<|E|/2$,
	\[|\ell \cap \{y \in \mathbb{F}^d: |\pi^y(E)| \leq M\}| \leq 2M.\]
\end{lemma}
\begin{proof}
Let $T$ be the set of $y\in \mathbb{F}^d$ such that $|\pi^y(E)| \leq M$ and $L$ be the set of lines that contain at least one point of $T$ and at least one point of $E$. Let
	\[R:= \{(y,x_1,x_2) \in (T\cap \ell) \times E \times E: y \neq x_1 \neq x_2 \text{ and } y, x_1, x_2 \text{ are co-linear}\}. \]
	Since each ordered pair $(x_1,x_2) \in E^2$ of distinct points in $E$ determines a line that intersects $\ell$ in at most one point, we have
	\begin{equation} \label{eq:upperBoundRForCollinear} |R| \leq |E|(|E|-1) < |E|^2. \end{equation}
	On the other hand, we have 
	\begin{align}
	|R| &= \sum_{y \in T\cap \ell} \sum_{\substack{\ell' \in L \\ y \in \ell'}} e(\ell')(e(\ell') - 1)\nonumber\\ &\geq |T\cap \ell|M^{-1}|E|^2 - |T\cap \ell|\,|E| = |T\cap \ell||E|^2M^{-1}(1-M|E|^{-1})\nonumber\\ &\geq 2^{-1}|T\cap \ell|\,|E|^2M^{-1}.\label{eq:lowerBoundRForCollinear}
	\end{align}
	Here, we use the Cauchy-Schwarz inequality and the fact that each point of $T$ is incident to at most $M$ lines of $L$.
	
	The result follows directly from (\ref{eq:upperBoundRForCollinear}) and (\ref{eq:lowerBoundRForCollinear}).
\end{proof}

Theorem \ref{th:3.3intro} follows easily from Lemmas \ref{lem:TOffALine} and \ref{lem:TOnALine}.

\begin{proof}[Proof of Theorem \ref{th:3.3intro}]
	Let $\ell$ be a line.
	Since there are at least $|E|/2$ points of $E$ that do not lie on $\ell$ and $M < |E|/4$, Lemma \ref{lem:TOnALine} implies that there are at most $2M$ points of $T$ contained in $\ell$.
	Consequently, Lemma \ref{lem:TOffALine} applied with $C = |E|M^{-1}$ and $k = 2M$ implies that
	\[|T| \leq (|E|M^{-1}-1)^{-1}2M|E| < 4M^2.  \]
\end{proof}

\subsection{Proof of Theorem \ref{prop1.6}}
As promised in the introduction, the proof of Theorem \ref{prop1.6} is very simple.

\begin{proof}[Proof of Theorem \ref{prop1.6}]
Let $T$ be the set of points such that $|\pi^y(E)| < |E|^{1/2}$, and let $M= 2^{-1}|E|^{1/2}-1$.
Let $L$ be the set of lines that contain at least one point of $T$ and at least one point of $E$.

Suppose, for contradiction, that $y,z \in T$ with $y \neq z$.
Let $L_y,L_z \subset L$ be those lines of $L$ incident to $y$ and $z$, respectively.
Let $\ell_M$ be the line that contains both $y$ and $z$.
Let $\ell \in L_y \setminus \{\ell_M\}$.
Then each point of $E \cap \ell$ is on a different line of $L_z$, hence $|\ell \cap E| \leq M$.
Since $|L_y \setminus \{\ell_M\}| \leq M$, we have that $\left |\bigcup_{\ell \in L_y \setminus \{\ell_M\}} (L_y \cap E) \right | \leq M^2$.
Since $M^2 < (1/4)|E|$, it must be the case that $|\ell_M \cap E| > (3/4)|E|$, contradicting our assumption.
\end{proof}

\section{Proofs of Theorems \ref{th:largeESC} and \ref{th:largeTSC}}

The proofs in this section use the Cauchy-Schwarz inequality together with several combinatorial properties of points and lines in $\mathbb{F}_q^d$.
In particular, each point is incident to $\binom{d}{1}_q$ lines, each line is incident to $q$ points, and $\mathbb{F}_q^d$ has $q^d$ points in total.

The proofs of Theorems \ref{th:largeESC} and \ref{th:largeTSC} are similar.
The basic plan of both proofs is to give upper and lower bounds on $\sum_{\ell \in L} e(\ell)t(\ell)$.

We use the same simple lower bound for both proofs.
Each pair $(x,y) \in E \times T$ is contained in exactly one line of $L$ if $x \neq y$, and at least one line of $L$ if $x = y$.
Hence,
\begin{equation}\label{eq:etLowerBound}
\sum_{\ell \in L} e(\ell)t(\ell) \geq |E|\,|T|.
\end{equation}

The upper bounds on $\sum_{\ell \in L} e(\ell)t(\ell)$ used in the two proofs of this section rely on different arguments, but both are inspired by work on incidence bounds for large sets in finite space.
These bounds were first investigated by Haemmers \cite{Ham}, and have had many recent applications to problems in arithmetic combinatorics and Erd\H{o}s-type questions over finite vector spaces \cite{YMRS, BIP, MPRRS, RS, RRS}.

The following lemma and its proof is nearly identical to Lemma 1 in \cite{MP2016}.
\begin{lemma}\label{lem:MP}
Let $E \subset \mathbb{F}_q^d$.
Then,
\begin{align}
    \label{eq:e2}&\sum_{\ell \in G} e(\ell)^2 = \binom{d}{1}_q |E| + |E|(|E|-1), \text{ and}\\
    \label{eq:eVar}&\sum_{\ell \in G} (e(\ell) - |E|q^{1-d})^2 \leq \binom{d}{1}_q|E|.
\end{align}
\end{lemma}
\begin{proof}
The following calculation yields (\ref{eq:e2}):
\begin{align*}
    \sum_{\ell \in G} e(\ell)^2 
    &= \sum_{\ell \in G} \sum_{(x,x') \in E^2} \ell(x) \ell(x') \\
    &= \sum_\ell \sum_{x \in E} \ell(x) + \sum_\ell \sum_{\substack{(x,x') \in E^2: \\x \neq x'}} \ell(x)\ell(x') \\
    &= \binom{d}{1}_q |E| + |E|(|E|-1).
\end{align*}
In the third line, we use the facts that each point is incident to exactly $\binom{d}{1}_q$ lines, and each pair of distinct points is contained in exactly one line.

We now show that (\ref{eq:eVar}) follows from (\ref{eq:e2}). Indeed, 
\begin{align*}
    \sum_{\ell \in G}(e(\ell) - |E|q^{-(d-1)})^2 &= \sum_{\ell \in G} e(\ell)^2 - 2 |E|q^{-(d-1)} \sum_{\ell \in G}e(\ell) + |G| |E|^2q^{-2(d-1)}  \\
    &= \sum_{\ell \in G} e(\ell)^2 - q^{-(d-1)}\binom{d}{1}_q |E|^2 \\
    &< \binom{d}{1}_q |E|.
\end{align*}
In the second line, we use the facts that each point is incident to exactly $\binom{d}{1}_q$ lines, and that $|G| = q^{d-1}\binom{d}{1}_q$.
\end{proof}

\subsection{Proof of Theorem \ref{th:largeESC}}
First we obtain an upper bound on $\sum_{\ell \in L} e(\ell)t(\ell)$.
\begin{lemma}\label{lem:largeEUpperBound}
We have
\begin{equation}\label{eq:largeEUpperBound}
    \sum_{\ell \in L} e(\ell)t(\ell) \leq (M q^{-(d-1)})|E|\,|T| + \sqrt{(M|T|+|T|^2)|E|\binom{d}{1}_q}.
\end{equation}
\end{lemma}
\begin{proof}
Since each point of $T$ is incident to at most $M$ lines of $L$, we have $\sum_{\ell \in L} t(\ell) \leq M|T|$, and hence
\begin{align}\sum_{\ell \in L} \nonumber e(\ell)t(\ell) & \leq  \frac{M|E|\,|T|}{q^{d-1}} + \sum_{\ell \in L}t(\ell)\left(e(\ell) - |E|q^{-(d-1)}\right) \\
\nonumber &\leq \frac{M|E|\,|T|}{q^{d-1}} + \sqrt{\sum_{\ell \in L}t(\ell)^2 \cdot \sum_{\ell \in L}\left(e(\ell) - |E|q^{-(d-1)}\right)^2 } \\
\label{eq:intermediateLargeE} &\leq \frac{M|E|\,|T|}{q^{d-1}} + \sqrt{\sum_{\ell \in L}t(\ell)^2 \cdot \sum_{\ell \in G}\left(e(\ell) - |E|q^{-(d-1)}\right)^2 }.
\end{align}
Lemma \ref{lem:MP} provides an upper bound for $\sum_{\ell \in G}\left(e(\ell) - |E|q^{-(d-1)}\right)^2$, so we only need to bound $\sum_{\ell \in L}t(\ell)^2$.
\begin{align}
    \nonumber \sum_{\ell \in L} t(\ell)^2 
    &= \sum_{\ell \in L} \sum_{(y,y') \in T^2} \ell(y) \ell(y') \\
    \nonumber &= \sum_{\ell \in L} t(\ell) + \sum_\ell \sum_{\substack{(y,y') \in T^2: \\y \neq y'}} \ell(y)\ell(y') \\
    \label{eq:boundT2ForLargeE}&\leq M |T| + |T|(|T|-1) < M|T| + |T|^2.
\end{align}
Combining (\ref{eq:boundT2ForLargeE}) with the bound $\sum_{\ell \in G}\left(e(\ell) - |E|q^{-(d-1)}\right)^2 \leq |E| \binom{d}{1}_q$ given by Lemma \ref{lem:MP} and (\ref{eq:intermediateLargeE}) completes the proof.
\end{proof}

We use Lemma \ref{lem:largeEUpperBound} to prove the following, which is slightly more general than Theorem \ref{th:largeESC}.

\begin{theorem}\label{th:largeE}
Let $E \subset \mathbb{F}_q^d$ and let $M$ be a positive integer.
Let $a= \binom{d}{1}_q|E|^{-1}$ and let $b = Mq^{1-d}$.
Let $C = (1-2b+b^2-a)^{-1}$.
If $C>0$, then
\begin{equation}\#\{y\in \mathbb{F}_q^d\colon |\pi^y(E)|\leq M\} <  C\binom{d}{1}_qM|E|^{-1}.\end{equation}
\end{theorem}

\begin{proof}
Combining the lower bound (\ref{eq:etLowerBound}) with the upper bound (\ref{eq:largeEUpperBound}) from Lemma \ref{lem:largeEUpperBound} leads quickly to the conclusion of the theorem:
\begin{align*}
    (1-2b+b^2)|E|^2|T|^2 &= (1-Mq^{-(d-1)})^2 |E|^2 |T|^2 \\
    &< (M|T|+|T|^2)|E|\binom{d}{1}_q \\
    &= M|T|\binom{d}{1}_q + a|E|^2|T|^2.
\end{align*}
Hence, 
\[|T| < C\binom{d}{1}_q M |E|^{-1},\]
as claimed.
\end{proof}

\begin{proof}[Proof of Theorem \ref{th:largeESC}]
Apply Theorem \ref{th:largeE} with $a=3^{-1}$ and $b=4^{-1}$.
\end{proof}

\subsection{Proof of Theorem \ref{th:largeTSC}}
To prove Theorem \ref{th:largeTSC}, we first need a refinement of Lemma \ref{lem:largeEUpperBound}. 
\begin{lemma}\label{lem:largeTUpperBound}
Let $a = |E|q^{1-d}$, let $b=M |E|^{-1}$, and let $c = \sqrt{(1-b)/a}$.
If $(1-b)/a>1$, then
\[\sum_{\ell \in L} e(\ell)t(\ell) < (b+ac)|E|\,|T| + (1-c^{-1})^{-1}|E|\sqrt{2\binom{d}{1}_q|T|}. \]
\end{lemma}
\begin{proof}
We divide the sum $\sum_{\ell \in L} e(\ell)t(\ell)$ into three parts:
\[\sum_{\ell \in L} e(\ell)t(\ell) = \sum_{\ell:e(\ell)= 1} e(\ell)t(\ell) + \sum_{\substack{\ell: e(\ell) \geq 2, \\t(\ell) \leq c|T|q^{1-d}}} e(\ell)t(\ell) + \sum_{\substack{\ell : e(\ell) \geq 2, \\t(\ell) > c|T|q^{1-d}}} e(\ell)t(\ell).\]
Bounding the first two terms is easy:
\begin{equation}\label{eq:L1}
    \sum_{\ell:e(\ell) = 1} t(\ell) \leq \sum_{y \in T} |\pi^y(E )| \leq M|T| = b|T|\,|E|,
\end{equation}
\begin{equation}\label{eq:L2}
    \sum_{\substack{\ell: e(\ell) \geq 2, \\t(\ell) \leq c|T|q^{1-d}}} e(\ell)t(\ell) \leq c|T|q^{1-d} \sum_{\ell: e(\ell) \geq 2} e(\ell) < c|T|q^{1-d}|E|^{2} = ac |T| \, |E|.
\end{equation}

Let $L'=\{\ell \in L: e(\ell) \geq 2, t(\ell) > c|T|q^{1-d}\}$.
We have
\begin{equation}\label{eq:L'CS}
\left(\sum_{\ell \in L'} e(\ell)t(\ell)\right)^2 \leq \sum_{\ell \in L'}t(\ell)^2 \sum_{\ell \in L'}e(\ell)^2.\end{equation}

We use the lower bound on $t(\ell)$ together with Lemma \ref{lem:MP} to bound $\sum_{\ell \in L'} t(\ell)^2$.
Denote $d = (1-c^{-1})^{-2}$.
Note that 
\[d^{-1}t(\ell)^2 = (t(\ell)-c^{-1}t(\ell))^2 \leq (t(\ell)-|T|q^{1-d})^2.\]
Together with (\ref{eq:eVar}) of Lemma \ref{lem:MP}, this yields
\begin{equation}\label{eq:L'T2}
\sum_{\ell \in L'}t(\ell)^2 = d \sum_{\ell \in L'} d^{-1}t(\ell)^2 \leq d \sum_{\ell \in L'}(t(\ell) - |T|q^{1-d})^2 < (1-c^{-1})^{-2}\binom{d}{1}_q |T|.
\end{equation}

To bound $\sum_{\ell \in L'} e(\ell)^2$, we use the fact that each point of $E$ is incident to fewer than $|E|$ lines of $L'$ to infer
\begin{equation}\label{eq:L'E2}
    \sum_{\ell \in L'} e(\ell)^2  < 2|E|^2,
\end{equation}

Taken together, (\ref{eq:L'CS}), (\ref{eq:L'T2}), and (\ref{eq:L'E2}) imply
\begin{equation}\label{eq:L'}
\left(\sum_{\ell \in L'} e(\ell)t(\ell)\right)^2 \leq \sum_{\ell \in L'}t(\ell)^2 \sum_{\ell \in L'}e(\ell)^2 < 2(1-c^{-1})^{-2}\binom{d}{1}_q|T| \, |E|^2. \end{equation}

Combining (\ref{eq:L1}), (\ref{eq:L2}), and (\ref{eq:L'}) leads immediately to the claimed conclusion.
\end{proof}

We remark that a very similar proof leads to a bound analogous to \ref{conti-2}.
The key idea that leads to our stronger result is to use the trivial bound \ref{eq:L'E2} instead of Lemma \ref{lem:MP} to bound $\sum_{\ell \in L'}e(\ell)^2$.

We use Lemma \ref{lem:largeTUpperBound} to obtain the following, which is slightly stronger than Theorem \ref{th:largeTSC}.

\begin{theorem}\label{th:largeT}
Let $E \subset \mathbb{F}_q^d$ and let $M$ be a positive integer.
Let $a = |E|q^{1-d}$, and let $b=M |E|^{-1}$.
Let $c = \sqrt{(1-b)/a}$.
Then,
\[\#\{y \in \mathbb{F}_q^d: |\pi^y(E )| \leq M\} <  2(1-b-ac)^{-2}(1-c^{-1})^{-2}\binom{d}{1}_q. \]
\end{theorem}
\begin{proof}
Lemma \ref{lem:largeTUpperBound} and (\ref{eq:etLowerBound}) together imply
\[|T|\,|E| < (b+ac)|T|\,|E| + (1-c^{-1})^{-1}|E|\sqrt{2 \binom{d}{1}_q|T|}, \]
which easily implies the claimed bound on $|T|$.
\end{proof}

\begin{proof}[Proof of Theorem \ref{th:largeTSC}]
Apply Theorem \ref{th:largeT} with $a=100^{-1}$ and $b=10^{-1}$.
\end{proof}

\section{Proofs Theorems \ref{radial-real-intro} and \ref{primefields2D}}
The proofs in this section are less elementary than those in the previous sections, and rely on incidence bounds.
The two proofs are identical in outline, but the calculations are simpler over the reals.

\subsection{Over the reals}
To proceed, we need to recall the following celebrated Szemer\'{e}di-Trotter theorem and one of its consequence on the number of $k$-rich lines. 
\begin{theorem}[\cite{Szemeredi}]\label{th:SzTr}
Let $E$ be a set of points and $L$ be a set of lines in $\mathbb{R}^2$.
Then,
\begin{equation}\label{eq:SzTrIncidenceBound}
\sum_{\ell \in L}e(\ell) = O(|E|^{2/3}|L|^{2/3} + |E| + |L|).
\end{equation}
Let $L_{\ge k}$ be the set of $k$-rich lines, i.e. lines with at least $k$ points from $E$, then
\begin{equation}\label{eq:SzTrRichLineBound}
|L_{\geq k}| = O(|E|^2k^{-3} + |E|k^{-1}).
\end{equation}
\end{theorem}

The following is a standard consequence, we include a proof for the reader's convenience.
\begin{lemma}\label{th:SzTrSquareBound}
There exist positive constants $c_1,c_2$ such that, for any set $E$ of points in $\mathbb{R}^2$,
\begin{equation}
    \sum_{c_1 \leq k \leq c_2|E|} k^2 |L_{=k}| < 10^{-1}|E|^2.
\end{equation}
\begin{proof}
\begin{align*}
    \sum_{c_1 \leq k \leq c_2|E|} k^2 |L_{=k}| &= \sum_{c_1 \leq k \leq c_2|E|} k^2 (|L_{\geq k}| - |L_{\geq k+1}|) \\
    &= \sum_{c_1 \leq k \leq c_2|E|} k^2 |L_{\geq k}| - \sum_{c_1 + 1 \leq k \leq c_2|E| + 1} (k-1)^2 |L_{\geq k}| \\
    &= \sum_{c_1 \leq k \leq c_2|E|} (k^2-(k-1)^2) |L_{=k}|+ c_1^2|L_{=c_1}| - (c_2|E|)^2|L_{=c_2 |E|+1}| \\
    &\leq \sum_{c_1 \leq k \leq c_2|E|-1} k O(|E|^2 k^{-3} + |E|k^{-1}) + O(c_1^{-1}|E|^2)\\
    &\leq \sum_{c_1 \leq k} O(|E|^2 k^{-2}) + \sum_{k \leq c_2|E|}O(|E|) + O(c_1^{-1}|E|^2)\\
    &\leq O(c_1^{-1}|E|^2) + O(c_2|E|^2).
\end{align*}
No matter what the constants hidden in the $O$-notation are, there are choices for $c_1$ and $c_2$ so that the right side is bounded by $10^{-1}|E|^2$, as claimed.
\end{proof}
\end{lemma}
\begin{theorem}\label{radial-real}
Let $c_1$ and $c_2$ be as in Lemma \ref{th:SzTrSquareBound}.
There exists a constant $0 <c_3<1$ such that the following holds.
For any finite $E \subset \mathbb{R}^2$ such that $|\ell \cap E| < c_2 |E|$ for each line $\ell$ and $M \leq c_3 |E|$, we have
\[\#\{y \in \mathbb{R}^2 : \pi^y(E) < M\} < O(M^2 E^{-1}). \]
\end{theorem}
\begin{proof}
As before, we proceed by bounding the sum $\sum_{\ell \in L}e(\ell)t(\ell)$ from above and below.
We use the same lower bound $\sum_{\ell \in L} e(\ell)t(\ell) \geq |E|\,|T|$ as before.
For the upper bound, we partition the sum as follows:
\begin{equation}
    \sum_{\ell \in L} e(\ell)t(\ell) \leq \sum_{\substack{\ell \in L: \\ e(\ell) < c_1}} e(\ell)t(\ell)  + \sum_{\substack{\ell \in L: \\ e(\ell) \geq c_1, t(\ell) > 1}} e(\ell)t(\ell)+\sum_{\substack{\ell \in L: \\ e(\ell) \geq c_1, t(\ell) = 1}}e(\ell)t(\ell).
\end{equation}
{\bf Bounding the first term:} This is easy, namely,
\begin{equation}\nonumber
    \sum_{\substack{\ell \in L: \\ e(\ell) < c_1}} e(\ell)t(\ell) = \sum_{y \in T} \sum_{\substack{\ell \in L:\\ y \in \ell}} e(\ell) < c_1 |M|\, |T| \leq c_1c_3|E|\,|T|.
\end{equation}

{\bf Bounding the second term:} By the Cauchy-Schwarz inequality,
\begin{equation}\label{30-e}
    \left(\sum_{\substack{\ell \in L: \\ e(\ell) \geq c_1, t(\ell) > 1}} e(\ell)t(\ell)\right)^2 \leq \sum_{\substack{\ell \in L: \\ e(\ell) \geq c_1, t(\ell) > 1}} e(\ell)^2 \sum_{\substack{\ell \in L: \\ e(\ell) \geq c_1, t(\ell) > 1}} t(\ell)^2.
\end{equation}
Lemma \ref{th:SzTrSquareBound} bounds the first sum on the right side by $10^{-1}|E|^2$.
To bound the second sum, we use the observation that each point of $T$ is incident to fewer than $T$ lines $\ell$ such that $t(\ell) > 1$:
\begin{align*}
    \sum_{\substack{\ell \in L: \\ t(\ell) > 1}}t(\ell)^2 &= \sum_{\ell \in L} \sum_{y_1 \neq y_2 \in T} \ell(y_1)\ell(y_2) + \sum_{\ell \in L} \sum_{y \in T} \ell(y)\\
    &= |T|(|T|-1) + \sum_{y \in T} \sum_{\ell \in L} \ell(y) \\
    &\leq |T|(|T|-1) + \sum_{y \in T} |T|  < 2|T|^2.
\end{align*}

Combining these leads to
\begin{equation}\nonumber
    \sum_{\substack{\ell \in L: \\ e(\ell) \geq c_1, t(\ell) > 1}} e(\ell)t(\ell) \leq 5^{-1/2}|E|\,|T|.
\end{equation}
{\bf Bounding the third term:} The key observation is that 
\[\#\{\ell \in L: t(\ell) = 1\} \leq M|T|.\]
Combined with the Szemer\'edi-Trotter theorem \ref{th:SzTr}, one has
\begin{equation}\nonumber
    \sum_{\substack{\ell \in L: \\ e(\ell) \geq c_1, t(\ell) = 1}}e(\ell)t(\ell) \leq \sum_{\substack{\ell \in L : \\ t(\ell) = 1}} e(\ell) = O((M|T|\,|E|)^{2/3} + M|T| + |E|).
\end{equation}

Combining the upper and lower bounds gives
\begin{equation}\nonumber
    |E|\,|T| \leq (c_1\,c_3 + 5^{-1/2})|E|\,|T| + O((M|E|\,|T|)^{2/3} + c_3|E|\,|T| + |E|).
\end{equation}
By choosing $c_3$ sufficiently small, we have
\begin{equation}\nonumber
    |E|\,|T| = O((M|E|\,|T|)^{2/3}),
\end{equation}
which leads directly to the desired conclusion.
\end{proof}
\subsection{Over prime fields}
In the prime field setting, we make use of a variant of the Stevens-De Zeeuw point-line incidence bound in \cite[Theorem 14]{lund}.
\bigskip
\begin{theorem}[\cite{frank}]\label{frank}
Let $P$ be a point set in $\mathbb{F}_p^2$ and $L$ be a set of lines in $\mathbb{F}_p^2$. If $|P|\le p^{8/5}$, then 
\[\sum_{\ell\in L}e(\ell)=O(|P|^{11/15}|L|^{11/15}+|P|+|L|).\]
Let $L_{\ge k}$ be the set of $k$-rich lines, i.e. lines with at least $k$ points from $P$, then 
\[|L_{\ge k}|=O(|P|^{11/4}k^{-\frac{15}{4}}+|P|k^{-1}).\]
\end{theorem}
We also have an analog of Lemma \ref{th:SzTrSquareBound}.

\begin{lemma}\label{th:SzTrSquareBound-ff}
There exists a positive constant $c$ such that, for any set $E$ of points in $\mathbb{R}^2$,
\begin{equation}
    \sum_{c^{-4/7}|E|^{3/7} \leq k \leq c|E|} k^2 |L_{=k}| < 10^{-1}|E|^2.
\end{equation}
\begin{proof}
Set $c_1:=c^{-4/7}|E|^{3/7}$. 
\begin{align*}
    \sum_{c_1 \leq k \leq c|E|} k^2 |L_{=k}| &= \sum_{c_1 \leq k \leq c|E|} k^2 (|L_{\geq k}| - |L_{\geq k+1}|) \\
    &= \sum_{c_1 \leq k \leq c|E|} k^2 |L_{\geq k}| - \sum_{c_1 + 1 \leq k \leq c|E| + 1} (k-1)^2 |L_{\geq k}| \\
    &= \sum_{c_1 \leq k \leq c|E|} (k^2-(k-1)^2) |L_{=k}|+ c_1|L_{=c_1}| - (c|E|)^2|L_{=c |E|+1}| \\
    &\leq \sum_{c_1 \leq k \leq c|E|-1} k O(|E|^{11/4} k^{-15/4} + |E|k^{-1}) + O(c_1^{-7/4}|E|^{11/4})+O(c_1|E|)\\
    &\leq \sum_{c_1 \leq k} O(|E|^{11/4} k^{-7/4}) + \sum_{k \leq c_2|E|}O(|E|) + O(c_1^{-1}|E|^2)\\
    &\leq O(c_1^{-7/4}|E|^{11/4})+O(c_1|E|) + O(c|E|^2).
\end{align*}
So we can choose $c$ small enough, so that the right side is bounded by $10^{-1}|E|^2$.
\end{proof}
\end{lemma}
\begin{proof}[Proof of Theorem \ref{primefields2D}]
Our argument is similar to that over the reals, for the sake of completeness, we reproduce all details here.

As above, we want to bound the sum $\sum_{\ell \in L}e(\ell)t(\ell)$ from above and below.
We have $\sum_{\ell \in L} e(\ell)t(\ell) \geq |E|\,|T|$.
For the upper bound, we partition the sum as follows:

\begin{equation}
    \sum_{\ell \in L} e(\ell)t(\ell) \leq \sum_{\substack{\ell \in L: \\ e(\ell) < c_1}} e(\ell)t(\ell)  + \sum_{\substack{\ell \in L: \\ e(\ell) \geq c_1, t(\ell) > 1}} e(\ell)t(\ell)+\sum_{\substack{\ell \in L: \\ e(\ell) \geq c_1, t(\ell) = 1}}e(\ell)t(\ell),
\end{equation}
where $c_1:=c^{-4/7}|E|^{3/7}$. 

{\bf Bounding the first term:} This is easy, namely,
\begin{equation}\nonumber
    \sum_{\substack{\ell \in L: \\ e(\ell) < c_1}} e(\ell)t(\ell) = \sum_{y \in T} \sum_{\substack{\ell \in L:\\ y \in \ell}} e(\ell) < c_1 |M|\, |T| \leq c|E|\,|T|.
\end{equation}

{\bf Bounding the second term:} By the Cauchy-Schwarz inequality,
\begin{equation}\label{30-eee}
    \left(\sum_{\substack{\ell \in L: \\ e(\ell) \geq c_1, t(\ell) > 1}} e(\ell)t(\ell)\right)^2 \leq \sum_{\substack{\ell \in L: \\ e(\ell) \geq c_1, t(\ell) > 1}} e(\ell)^2 \sum_{\substack{\ell \in L: \\ e(\ell) \geq c_1, t(\ell) > 1}} t(\ell)^2.
\end{equation}
Lemma \ref{th:SzTrSquareBound-ff} bounds the first sum on the right side by $10^{-1}|E|^2$.
To bound the second sum, we use the observation that each point of $T$ is incident to fewer than $T$ lines $\ell$ such that $t(\ell) > 1$:
\begin{align*}
    \sum_{\substack{\ell \in L: \\ t(\ell) > 1}}t(\ell)^2 &= \sum_{\ell \in L} \sum_{y_1 \neq y_2 \in T} \ell(y_1)\ell(y_2) + \sum_{\ell \in L} \sum_{y \in T} \ell(y)\\
    &= |T|(|T|-1) + \sum_{y \in T} \sum_{\ell \in L} \ell(y) \\
    &\leq |T|(|T|-1) + \sum_{y \in T} |T|  < 2|T|^2.
\end{align*}

Combining these leads to
\begin{equation}\nonumber
    \sum_{\substack{\ell \in L: \\ e(\ell) \geq c_1, t(\ell) > 1}} e(\ell)t(\ell) \leq 5^{-1/2}|E|\,|T|.
\end{equation}
{\bf Bounding the third term:} The key observation is that 
\[\#\{\ell \in L: t(\ell) = 1\} \leq M|T|.\]
Combined with the Theorem \ref{frank}, one has
\begin{equation}\nonumber
    \sum_{\substack{\ell \in L: \\ e(\ell) \geq c_1, t(\ell) = 1}}e(\ell)t(\ell) \leq \sum_{\substack{\ell \in L : \\ t(\ell) = 1}} e(\ell) = O((M|T|\,|E|)^{11/15} + M|T| + |E|)=O\left((M|T|\,|E|)^{11/15}+|E|\right)+\frac{|T|\,|E|}{10}.
\end{equation}

Combining the upper and lower bounds gives
\begin{equation}\nonumber
    |E|\,|T| \leq (c + 5^{-1/2})|E|\,|T| + O\left((M|T|\,|E|)^{11/15}+|E|\right)+\frac{|T|\,|E|}{10}.
\end{equation}
By choosing $c$ sufficiently small, we have
\begin{equation}\nonumber
    |E|\,|T| = O((M|E|\,|T|)^{11/15}),
\end{equation}
which leads directly to the desired conclusion.

\end{proof}

\nonumber
\section{Acknowledgements}
We would like to thank Bochen Liu and Tuomas Orponen for many discussions about the results in the continuous setting. 

B. Lund was supported by the Institute for Basic Science (IBS-R029-C1). T. Pham would like to thank to the VIASM for the hospitality and for the excellent working condition. 
 

 \bibliographystyle{amsplain}

\end{document}